\documentclass{amsart}  

\usepackage{enumitem}
\usepackage{hyperref}

\newcommand*{\mailto}[1]{\href{mailto:#1}{\nolinkurl{#1}}}

\newtheorem{theorem}{Theorem}[section]
\newtheorem{corollary}[theorem]{Corollary}
\newtheorem{lemma}[theorem]{Lemma}

\theoremstyle{definition}

\newtheorem{example}[theorem]{Example}
\newtheorem{remark}[theorem]{Remark}
\newtheorem*{CP}{Coupling problem}

\newcommand{\R}{{\mathbb R}}
\newcommand{\N}{{\mathbb N}}

\newcommand{\C}{{\mathbb C}}

\newcommand{\E}{\mathrm{e}}
\newcommand{\I}{\mathrm{i}}

\newcommand{\res}{\mathrm{res}}
\newcommand{\im}{\mathrm{Im}}
\newcommand{\re}{\mathrm{Re}}

\newcommand{\ledot}{\,\cdot\,}
\newcommand{\redot}{\cdot\,}

\newcommand{\oo}{o}

\numberwithin{equation}{section}

\begin{document}

\title[Long-time asymptotics for the Korteweg--de Vries equation]{Long-time asymptotics for the Korteweg--de Vries equation with integrable reflectionless initial data}

\dedicatory{Dedicated to Fritz Gesztesy on the occasion of his 70th birthday}

\author[J.\ Eckhardt]{Jonathan Eckhardt}
\address{Department of Mathematical Sciences\\ Loughborough University\\ Epinal Way\\ Loughborough\\ Leicestershire LE11 3TU \\ UK}
\email{\mailto{J.Eckhardt@lboro.ac.uk}}


\keywords{Korteweg--de Vries equation, long-time asymptotics, generalized soliton solutions}
\subjclass[2020]{Primary 37K40, 35Q53; Secondary 37K15, 34L25}  

\begin{abstract}
 We show that solutions of the Korteweg--de Vries equation with reflectionless integrable initial data decompose into a (in general infinite) linear superposition of solitons after long enough time. 
 The proof is based on a representation of reflectionless integrable potentials in terms of solutions to symmetric coupling problems for entire functions. 
\end{abstract}

\maketitle

\section{Introduction}

The purpose of this article is to establish long-time asymptotics for solutions of the Korteweg--de Vries equation
\begin{align}
  q_t = 6qq_x - q_{xxx}
\end{align}
on the real line with initial data that is integrable and reflectionless. 
Under the more restrictive assumption that the initial data has a finite first moment, these solutions are precisely the $N$-soliton solutions given explicitly by 
\begin{align}\label{eqnKM}
  q(x,t) = - 2 \frac{d^2}{dx^2} \log\det\biggl(\delta_{ij} + \frac{\gamma_i \gamma_j \E^{4(\kappa_i^3+\kappa_j^3) t-(\kappa_i+\kappa_j)x}}{\kappa_i +\kappa_j} \biggr)_{i, j =1}^N
\end{align}
for some positive parameters $\kappa_1,\ldots,\kappa_N$ and $\gamma_1,\ldots,\gamma_N$. 
For each fixed time $t\in\R$, the function $q(\ledot,t)$ is known to be reflectionless when considered as the potential of the (maximally defined) one-dimensional Schr\"odinger operator
\begin{align}
    - \frac{d^2}{dx^2} + q(\ledot,t)
\end{align}
in $L^2(\R)$.
Moreover, the constants $-\kappa_1^2,\ldots,-\kappa_N^2$ are precisely the negative eigenvalues of this Schr\"odinger operator, independent of time $t$. 

Asymptotically for large times $t$, these $N$-soliton solutions will decompose into a linear superposition of $N$ single solitons. 
More precisely, one has~\cite{ta72, wato72} that  
 \begin{align}\label{eqnqLT}
  q(x,t) = \sum_{n=1}^N \frac{-2\kappa_n^2}{\cosh^2(\kappa_n x-4\kappa_n^3 t - \xi_n)} + \oo(1) 
 \end{align}
  uniformly in $x\in\R$ as $t\rightarrow\infty$, 
 where the phase shifts $\xi_1,\ldots,\xi_N$ are given explicitly in terms of the parameters $\kappa_1,\ldots,\kappa_N$ and $\gamma_1,\ldots,\gamma_N$.
 In fact, much more refined asymptotics can be obtained by now for solutions of the Korteweg--de Vries equation with sufficient spatial decay via the nonlinear steepest descent method for Riemann--Hilbert problems of Deift and Zhou~\cite{dezh93} (see~\cite{deitzh93} for a survey on the history of this problem and~\cite{grte09} for an expository introduction) by representing the solution $q$ in terms of solutions to certain Riemann--Hilbert problems. 

 Various ways have been suggested to generalize classical reflectionless potentials and their corresponding $N$-soliton solutions.
 For example, Marchenko~\cite{ma91} considered locally uniform limits of classical reflectionless potentials and Gesztesy, Karwowski and Zhao~\cite{gekazh92} obtained solutions by passing to the limit $N\rightarrow\infty$ in~\eqref{eqnKM} under suitable conditions on the parameters. 
 Here we will consider the class of integrable potentials (in which case it is still possible to define a reflection coefficient using Jost solutions) that are reflectionless. 
 This class of potentials and corresponding solutions to the Korteweg--de Vries equation have been studied recently by Hryniv, Melnyk and Mykytyuk~\cite{hrmemy21}, where a complete solution of the corresponding inverse spectral problem was obtained. 
 In this case, the number of negative eigenvalues of the underlying Schr\"odinger operator is not finite in general and the corresponding solutions to the Korteweg--de Vries equation are therefore expected to decompose into a potentially infinite superposition of solitons. 
 
 We will prove in Section~\ref{secGSS} that the long-time asymptotics in~\eqref{eqnqLT} indeed continue to hold for these solutions, where in the case of infinitely many eigenvalues uniform convergence of the sum on the right-hand side is guaranteed by a Lieb--Thirring inequality~\cite{we96}. 
  It appears that the only results on long-time asymptotics for solutions of this kind have been derived by Novoksh\"enov~\cite{no92} by means of the dressing method under several additional involved conditions on the spectral data. 
  Due to the lack of sufficiently fast spatial decay of these solutions and therefore the potential presence of infinitely many eigenvalues, using a representation of the solution $q$ in terms of solutions to Riemann--Hilbert problems and subsequently deploying available asymptotic methods does not seem to readily apply here.
  In fact, in our situation the Riemann--Hilbert problems (one can use a modification of the ones in~\cite[Section~2]{grte09} for example) would be required to allow infinitely many pole conditions. 
  On the other side, since we are restricted to the reflectionless case, the jump condition along the real axis actually disappears so that one ends up with discrete Riemann--Hilbert problems similar to the ones considered by Borodin~\cite{bo03, bo00}, however with poles accumulating at a finite point instead of infinity (which complicates the uniqueness aspect of solutions considerably). 
  Instead of such discrete Riemann--Hilbert problems, we will work here with essentially equivalent {\em symmetric coupling problems for entire functions}~\cite{UniSolCP, CouplingProblem}. 
  That reflectionless integrable potentials can always be represented in terms of solutions to symmetric coupling problems will be presented in Section~\ref{secRI}. 
 The key tool we then need is a result about asymptotic behavior of solutions to these coupling problems that will be derived in Section~\ref{secSCP}.
  Our proof is elementary\footnote{We note that it may appear that the proof requires uniqueness of coupling problems established in~\cite{UniSolCP}, which uses tools that are not elementary. 
However, this is not actually necessary as it is enough to know that the solutions in Example~\ref{exaSoliton} are unique, which can be established easily.} (based on the inverse scattering transform) and uses little apart from some basic complex analysis, but does not yield the strong asymptotics obtained via the nonlinear steepest descent method for solutions of the Korteweg--de Vries equation with stronger spatial decay.

\section{Symmetric coupling problems}\label{secSCP}

Let $\sigma$ be a discrete set of nonzero reals such that 
\begin{align}
 \sum_{\lambda\in\sigma} \frac{1}{|\lambda|} < \infty
\end{align}
and define the real entire function $W$ of exponential type zero by
\begin{align}
 W(z) = \prod_{\lambda\in\sigma} \biggl(1-\frac{z}{\lambda}\biggr).
\end{align}
For a sequence $\eta\in\hat{\R}^\sigma$ (called {\em coupling constants} or {\em data}), where $\hat{\R} = \R\cup\{\infty\}$ is the one-point compactification of $\R$, we consider the following task. 

\begin{CP}
Find a pair of real entire functions $(\Phi_-,\Phi_+)$ of exponential type zero such that the three conditions listed below are satisfied:
\begin{enumerate}[label=\emph{(\roman*)}, leftmargin=*, widest=W]
\item[(C)] {\em Coupling condition}: For each $\lambda\in\sigma$ one has\footnote{This condition should be understood as $\Phi_+(\lambda)=0$ when $\eta(\lambda) = \infty$.} 
\begin{align*}
 \Phi_-(\lambda) = \eta(\lambda) \Phi_+(\lambda).
\end{align*}
\item[(G)] {\em Growth and positivity condition}: For every $z\in\C$ with $\im\, z>0$ one has 
\begin{align*}
  \im \biggl( \frac{z \Phi_-(z) \Phi_+(z)}{W(z)}\biggr) \geq 0. 
\end{align*}
\item[(N)] {\em Normalization condition}: At zero one has 
\begin{align*}
 \Phi_-(0) = \Phi_+(0) = 1.
\end{align*}
\end{enumerate}
\end{CP}

If a pair $(\Phi_-,\Phi_+)$ is a solution of the coupling problem with data $\eta$, then the growth and positivity condition~(G) means that the function
\begin{align}\label{eqnGHN}
 \frac{z \Phi_-(z) \Phi_+(z)}{W(z)} 
\end{align}
is a meromorphic Herglotz--Nevanlinna function.
First of all, this guarantees that all zeros of the functions $\Phi_-$ and $\Phi_+$ are real. 
It furthermore entails that the zeros of the function in the numerator of~\eqref{eqnGHN} and the zeros of the function in the denominator of~\eqref{eqnGHN} are interlacing (after possible cancelations).
From this we are able to conclude that the functions $\Phi_-$ and $\Phi_+$ are actually of genus zero and admit the upper bound 
\begin{align}\label{eqnPhiBound}
 |\Phi_\pm(z)| \leq \prod_{\lambda\in\sigma} \biggl( 1 + \frac{|z|}{|\lambda|} \biggr),
\end{align}
 which is independent of the coupling constants $\eta$.  
Condition~(G) also implies that the residues at all poles of the function in~\eqref{eqnGHN} are negative. 
Together with the coupling condition~(C), this shows that        
\begin{align}\label{eqnetaphisign}
   \frac{\eta(\lambda)\Phi_+(\lambda)^2}{\lambda W'(\lambda)} \leq 0 
\end{align}
for all $\lambda\in\sigma$ with $\eta(\lambda)\in\R$. 
Unless it happens that $\lambda$ is a zero of $\Phi_+$, this constitutes the necessary restriction 
 \begin{align}\label{eqnCCcon}
   \frac{\eta(\lambda)}{\lambda W'(\lambda)} \leq 0 
 \end{align}
on the sign of the coupling constant $\eta(\lambda)$ in order for a solution of the coupling problem with data $\eta$ to exist. 
Roughly speaking, the coupling constants are expected to have alternating signs, beginning with non-negative ones for those corresponding to the smallest (in modulus) positive and negative element of $\sigma$. 
Even though this condition on the coupling constants is not necessary for solutions to exist in general, it appears to be reasonable and is satisfied in all applications so far.
 Indeed, it is always sufficient to guarantee solvability of the corresponding coupling problem.
More precisely, the main result in~\cite{UniSolCP} states that the coupling problem with data $\eta$ has a solution if the coupling constants satisfy~\eqref{eqnCCcon} for all $\lambda\in\sigma$ for which $\eta(\lambda)$ is finite.
Furthermore, solutions are always unique and depend continuously on the given data, when the coupling constants are endowed with the product topology. 

In the context of reflectionless integrable potentials, a particular kind of coupling problem arises: 
A coupling problem is called {\em symmetric} if the set $\sigma$ is symmetric about the imaginary axis and the coupling constants $\eta$ satisfy 
\begin{align}
 \eta(-\lambda)=\eta(\lambda)^{-1}
\end{align}
 for each $\lambda\in\sigma$, where $x\mapsto x^{-1}$ is defined on $\hat{\R}$ by continuous extension.

\begin{lemma}
  If a pair $(\Phi_-,\Phi_+)$ is the solution of a symmetric coupling problem, then one has $\Phi_-(z) = \Phi_+(-z)$. 
\end{lemma}

\begin{proof}
  For a solution $(\Phi_-,\Phi_+)$ of a symmetric coupling problem, the pair $(\Psi_-,\Psi_+)$ defined by $\Psi_-(z) = \Phi_+(-z)$ and $\Psi_+(z) = \Phi_-(-z)$ is again a solution of the same coupling problem. 
  Now the claim follows from uniqueness of solutions to coupling problems established in~\cite{UniSolCP}.   
\end{proof}

In the case when the set $\sigma$ consists of a single pair of points, it is possible to solve symmetric coupling problems explicitly.  
We will see that this corresponds to reflectionless integrable potentials with precisely one negative eigenvalue. 

\begin{example}\label{exaSoliton}
 Suppose that $\sigma = \lbrace -\lambda_0, \lambda_0 \rbrace$ for some positive $\lambda_0\in\R$, so that 
 \begin{align}
   W(z) = 1-\frac{z^2}{\lambda_0^2}. 
 \end{align}
 For such a set $\sigma$, a symmetric coupling problem with data $\eta\in\hat{\R}^\sigma$ is solvable if and only if the coupling constant $\eta(\lambda_0)$ is not a negative real number. 
 In this case, the solution $(\Phi_-,\Phi_+)$ of the symmetric coupling problem is given by
 \begin{align}
  \Phi_\pm(z) = 1 \pm \frac{z}{\lambda_0} \frac{1-\eta(\lambda_0)}{1+\eta(\lambda_0)},
 \end{align}
 where the last fraction has to be interpreted as minus one when $\eta(\lambda_0)$ is not finite.
 A straightforward computation then yields the identity (here, the last fraction has to be interpreted as zero when $\eta(\lambda_0)$ is equal to zero or infinity)    
  \begin{align}
    \frac{\Phi_-(z)\Phi_+(z)}{W(z)} & 
     = 1 + \frac{z^2}{\lambda_0^2-z^2} \frac{1}{\cosh^2\bigl(\frac{1}{2}\log\eta(\lambda_0)\bigr)},
 \end{align}
 hinting at the connection to reflectionless integrable potentials.
\end{example}

Other simple cases of symmetric coupling problems that can be solved explicitly are ones for which each of the coupling constants is either zero or infinity. 

\begin{example}\label{exaCPzeroinf}
 Consider a symmetric coupling problem with data $\eta\in\hat{\R}^\sigma$ such that $\eta(\lambda)$ is either zero or infinity for each $\lambda\in\sigma$.
 In this case, the solution $(\Phi_-,\Phi_+)$ of the symmetric coupling problem is given by
 \begin{align}
  \Phi_\pm(z) = \prod_{\lambda\in\sigma_0}\biggl(1\pm\frac{z}{\lambda}\biggr), 
 \end{align}
 where $\sigma_0=\{\lambda\in\sigma\,|\, \eta(\lambda) = 0\}$, so that one has 
  \begin{align}
    \frac{\Phi_-(z)\Phi_+(z)}{W(z)} = 1.
 \end{align}
\end{example}

In conjunction with continuous dependence of solutions on the coupling constants established in~\cite{UniSolCP}, Example~\ref{exaCPzeroinf} yields the following result that has also been proved in a different way in~\cite[Corollary~3.2]{CouplingProblem} or can be deduced as a special case of Theorem~\ref{thmAA} below. 

\begin{corollary}\label{corAA}
  For every $k\in\N$, let $\eta_k\in\hat{\R}^\sigma$ be such that the coupling problem with data $\eta_k$ is symmetric and has a solution $(\Phi_{k,-},\Phi_{k,+})$. 
 If $\eta_k(\lambda)$ converges to zero or infinity as $k\rightarrow\infty$ for each $\lambda\in\sigma\cap(0,\infty)$, then one has  
 \begin{align}
  \frac{\Phi_{k,-}(z)\Phi_{k,+}(z)}{W(z)} \rightarrow 1
 \end{align}
 locally uniformly for all $z\in\C\backslash\sigma$. 
\end{corollary}

In order to derive uniform long-time asymptotics for solutions of the Korteweg--de Vries equation with reflectionless integrable initial data later, we will need the following asymptotics for solutions to symmetric coupling problems in the situation when all but one pair of coupling constants tend to zero or infinity. 
Just like in Example~\ref{exaSoliton}, we are going to interpret the fraction 
\begin{align}
\frac{1}{\cosh^2\bigl(\frac{1}{2}\log|\eta|\bigr)}
\end{align}
 as a function in $\eta$ continuously extended to $\hat{\R}$. 

\begin{theorem}\label{thmAA}
  For every $k\in\N$, let $\eta_k\in\hat{\R}^\sigma$ be such that the coupling problem with data $\eta_k$ is symmetric and has a solution $(\Phi_{k,-},\Phi_{k,+})$. 
 If there is a partition $\sigma_0 \cup \lbrace\lambda_0\rbrace \cup \sigma_\infty$ of $\sigma\cap(0,\infty)$  such that   
 \begin{align}\label{eqnConvCC}
    \eta_{k}(\lambda) \rightarrow  \begin{cases} 0, & \text{for all }\lambda\in\sigma_0, \\ \infty, &  \text{for all }\lambda\in\sigma_\infty, \end{cases} 
 \end{align}
 as $k\rightarrow\infty$, then  one has the asymptotics 
 \begin{align}\label{eqnConv}
  \frac{\Phi_{k,-}(z)\Phi_{k,+}(z)}{W(z)} = 1 + \frac{z^2}{\lambda_0^2-z^2} \frac{1}{\cosh^2\bigl(\frac{1}{2}\log|\hat{\eta}_k|\bigr)} + \oo(1)
 \end{align}
  locally uniformly for all $z\in\C\backslash\sigma$, where the $\hat{\eta}_k\in\hat{\R}$ are given by 
 \begin{align}
  \hat{\eta}_{k} = \eta_{k}(\lambda_0) \prod_{\lambda\in\sigma_0} \frac{\lambda + \lambda_0}{\lambda - \lambda_0} \prod_{\lambda\in\sigma_\infty} \frac{\lambda - \lambda_0}{\lambda + \lambda_0}
 \end{align}
 when $\eta_k(\lambda_0)$ is finite and by $\hat{\eta}_k = \infty$ otherwise. 
\end{theorem}

\begin{proof}
 By means of a compactness argument (using the uniform bound in~\eqref{eqnPhiBound} and Montel's theorem), we may assume that the entire functions $\Phi_{k,-}$ and $\Phi_{k,+}$ converge locally uniformly and that the coupling constants $\eta_k(\lambda_0)$ converge to some $\eta_\infty$.
The respective limits of $\Phi_{k,-}$ and $\Phi_{k,+}$ are real entire functions of exponential type zero in view of~\eqref{eqnPhiBound} and will be denoted by $\Psi_-$ and $\Psi_+$. 
Due to the assumption~\eqref{eqnConvCC} and the coupling condition obeyed by the pairs $(\Phi_{k,-},\Phi_{k,+})$, we see that $\Psi_\pm(\lambda)=0$ whenever $\mp\lambda\in\sigma_0$ or $\pm\lambda\in\sigma_\infty$. 
One is therefore able to define real entire functions $P_-$ and $P_+$ of exponential type zero (see~\cite[Section~I.9]{le80}) such that  
\begin{align*}
 \Psi_\pm(z) = P_\pm(z) \prod_{\lambda\in\sigma_0} \biggl( 1\pm\frac{z}{\lambda} \biggr) \prod_{\lambda\in\sigma_\infty} \biggl( 1\mp\frac{z}{\lambda} \biggr).
\end{align*}
It follows that the pair $(P_-,P_+)$ satisfies the coupling conditions
\begin{align*}
 P_-(\lambda_0) & = \hat{\eta}_\infty  P_+(\lambda_0), & P_-(-\lambda_0) & = \hat{\eta}_\infty^{-1}  P_+(-\lambda_0),
\end{align*}
where the coupling constant $\hat{\eta}_\infty\in\hat{\R}$ is given by 
\begin{align*}
 \hat{\eta}_\infty = \eta_\infty \prod_{\lambda\in\sigma_0} \frac{\lambda + \lambda_0}{\lambda - \lambda_0} \prod_{\lambda\in\sigma_\infty} \frac{\lambda - \lambda_0}{\lambda + \lambda_0}  
\end{align*}
 when $\eta_\infty$ is finite and by $\hat{\eta}_\infty = \infty$ otherwise. 
Moreover, one has that 
\begin{align*}
  \im\biggl(zP_-(z)P_+(z) \biggl(1-\frac{z^2}{\lambda_0^2}\biggr)^{-1}\biggr) = \im\biggl(\frac{z\Psi_-(z) \Psi_+(z)}{W(z)} \biggr) \geq 0
\end{align*}
for all $z\in\C$ with $\im\,z >0$. 
As the functions $P_-$ and $P_+$ are clearly normalized at zero, we conclude that the pair $(P_-,P_+)$ is the unique solution of a symmetric coupling problem as in Example~\ref{exaSoliton}.  
Since this can be solved explicitly, we get that 
\begin{align*}
 \frac{\Phi_{k,-}(z) \Phi_{k,+}(z)}{W(z)} \rightarrow  1 + \frac{z^2}{\lambda_0^2-z^2} \frac{1}{\cosh^2\bigl(\frac{1}{2}\log|\hat{\eta}_\infty|\bigr)}
\end{align*}
 locally uniformly for all $z\in\C\backslash\sigma$. 
 Because the $\hat{\eta}_k$ converge to $\hat{\eta}_\infty$, one also has  
\begin{align*}
 \frac{z^2}{\lambda_0^2-z^2} \frac{1}{\cosh^2\bigl(\frac{1}{2}\log|\hat{\eta}_k|\bigr)} \rightarrow \frac{z^2}{\lambda_0^2-z^2} \frac{1}{\cosh^2\bigl(\frac{1}{2}\log|\hat{\eta}_\infty|\bigr)} 
\end{align*}
 locally uniformly for all $z\in\C\backslash\sigma$, which yields the claimed asymptotics.  
\end{proof}

\section{Reflectionless integrable potentials}\label{secRI}

We are next going to demonstrate that reflectionless integrable potentials can be represented in terms of solutions to symmetric coupling problems, where the set $\sigma$ is determined by the negative eigenvalues of the corresponding Schr\"odinger operator and the data is given in terms of associated norming constants. 
To this end, we briefly recall the most basic facts about scattering theory for one-dimensional Schr\"odinger operators with integrable potentials first. 
These can be found in references like~\cite{detr79, ma11, ya10}, although they usually suppose slightly stronger decay assumptions on the potential.
 However, the minimal facts we require here continue to hold when the potential is merely integrable and we refer to~\cite{hrmy21} and~\cite[Section~2]{hrmemy21} in particular, which contain precisely what we are going to need in this section. 

Given a real-valued potential $q\in L^1(\R)$, for each $k$ in the open upper complex half-plane $\C_+$ there is a unique solution $f_\pm(k,\redot)$, called the {\em right/left Jost solution}, of the differential equation 
\begin{align}
  -f'' + qf = k^2 f
\end{align}
with the spatial asymptotics 
\begin{align}
  f_\pm(k,x) \sim \E^{\pm\I kx}
\end{align}
as $x\rightarrow\pm\infty$. 
In particular, the solution $f_\pm(k,\redot)$ is the unique (up to constant multiples) nontrivial solution that is square integrable near $\pm\infty$. 
With these Jost solutions, the {\em transmission coefficient} $T$ can be defined by 
\begin{align}
  T(k) = \frac{2\I k}{f_+'(k,x)f_-(k,x) - f_+(k,x)f_-'(k,x)},
\end{align}
 which is a meromorphic function on $\C_+$ with a continuous extension to $\R$ except for possibly zero. 
All poles of $T$ lie on the imaginary axis and can be enumerated by $\I\kappa_1,\I\kappa_2,\ldots$ for some positive $\kappa_n$. 
The number of poles may be finite or infinite and will be denoted by $N\in\N\cup\{0,\infty\}$. 
In case there are infinitely many poles, they satisfy 
\begin{align}
  \sum_{n=1}^\infty \kappa_n < \infty, 
\end{align}
so that $\kappa_n\rightarrow0$ necessarily and we may assume that the $\kappa_n$ are strictly decreasing.
For each pole $\I\kappa_n$, the Jost solutions $f_-(\I\kappa_n,\redot)$ and $f_+(\I\kappa_n,\redot)$ are linearly dependent with 
\begin{align}\label{eqnJostCoup}
  f_-(\I\kappa_n,x) = c_n f_+(\I\kappa_n,x)
\end{align}
for some nonzero constant $c_n\in\R$. 
In particular, this guarantees that the Jost solutions are square integrable in this case, so that $-\kappa_n^2$ is a negative eigenvalue of the maximally defined Schrödinger operator 
\begin{align}
  - \frac{d^2}{dx^2} + q
\end{align}
in $L^2(\R)$. 
Conversely, every negative eigenvalue of this operator corresponds to a pole of $T$ in this way. 
Even though we will not need them here, let us mention that the usual associated right/left norming constants $\gamma_{\pm,n}>0$ are defined by 
\begin{align}
  \gamma_{\pm,n}^{-2} = \int_\R f_\pm(\I\kappa_n,x)^2 dx.
\end{align}
Since they are related to the constants $c_n$ via the transmission coefficient $T$ by 
\begin{align}\label{eqnCoupNorm}
  c_n^{\mp 1} \gamma_{\pm,n}^{2}= -\I\,\res_{\I\kappa_n} T,
\end{align}
the norming constants $\gamma_{\pm,n}$ can be used interchangeably instead of the constants $c_n$ (assuming that the transmission coefficient $T$ is known). 

In this article, we are concerned with {\em reflectionless integrable potentials}, which may be defined as follows: 
A {\em reflectionless integrable potential} is a real-valued function $q\in L^1(\R)$ such that the corresponding transmission coefficient $T$ satisfies $|T(k)|=1$ for all nonzero real $k$.
It is known (see~\cite[Corollary~5.3]{hrmemy21}) that $q$ has a continuous representative in this case, which we are going to use in the following. 

\begin{theorem}
 Let $q$ be a reflectionless integrable potential and define the discrete set $\sigma$ of nonzero reals by 
 \begin{align}
   \sigma = \biggl\{-\frac{1}{\kappa_n},\,\frac{1}{\kappa_n}\,\bigg|\, n=1,\ldots,N\biggr\}.
 \end{align}
 For each given point $x\in\R$ one has  
\begin{align}\label{eqnqRep}
  q(x) = - \frac{d^2}{dz^2} \frac{\Phi_-(z)\Phi_+(z)}{W(z)}\bigg|_{z=0},
\end{align}
where the pair $(\Phi_-,\Phi_+)$ is the solution of the symmetric coupling problem with data $\eta_x$ given by 
\begin{align}\label{eqnCCetax}
  \eta_x\biggl(\frac{1}{\kappa_n}\biggr) & = c_n \E^{-2\kappa_n x}.  
\end{align} 
\end{theorem}

\begin{proof}
For a reflectionless integrable potential $q$, the transmission coefficient $T$ is given by the Blaschke product 
\begin{align*}
   T(k) = \prod_{n=1}^N  \frac{k+\I\kappa_n}{k-\I\kappa_n}
\end{align*}
and when $x\in\R$ is fixed, the Jost functions $f_\pm(\ledot,x)$ can be factorized (see~\cite[Section~5.5]{hrmemy21}) as
\begin{align*}
 f_\pm(k,x) = \E^{\pm\I k x} \prod_{n=1}^N  \frac{k\mp\I\mu_n(x)}{k+\I\kappa_n},
\end{align*}
where $\mu_1(x),\mu_2(x),\ldots$ is a sequence of nonzero real numbers with 
\begin{align*}
  \sum_{n=1}^N |\mu_n(x)| < \infty. 
\end{align*} 
This allows us to define the pair of real entire functions $(\Phi_-,\Phi_+)$ of exponential type zero by 
\begin{align*}
  \Phi_\pm(z) = \prod_{n=1}^N (1\mp z\mu_n(x)),
\end{align*}
so that for $\re\,z>0$ one has the relation
\begin{align*}
  \Phi_\pm(z) = \E^{\pm \frac{x}{z}} f_\pm(\I/z,x) \prod_{n=1}^N (1+z\kappa_n).
\end{align*}
From this relation and property~\eqref{eqnJostCoup} of the Jost solutions, it follows that the pair $(\Phi_-,\Phi_+)$ satisfies the coupling condition 
\begin{align*}
 \Phi_-(\lambda) = \eta_x(\lambda) \Phi_+(\lambda) 
\end{align*}
for all $\lambda\in\sigma$. 
Moreover, the relation above also implies that  
\begin{align}\label{eqnDGreen}
  \frac{z\Phi_-(z)\Phi_+(z)}{W(z)} = -2 \biggl(\frac{f_+'(\I/z,x)}{f_+(\I/z,x)} - \frac{f_-'(\I/z,x)}{f_-(\I/z,x)}\biggr)^{-1}
\end{align}
for $\re\,z>0$, from which we are able to conclude that the left-hand side has non-negative imaginary part for all $z\in\C$ with $\im\,z>0$.
In fact, the quotients on the right-hand side are the half-line Weyl--Titchmarsh functions, which are known to be Herglotz--Nevanlinna functions (alternatively, one can also use~\cite[Lemma~1.1]{humBre16} for example).
As the functions $\Phi_-$ and $\Phi_+$ are clearly normalized at zero, we conclude that the pair $(\Phi_-,\Phi_+)$ is the solution of the coupling problem with data $\eta_x$. 
In order to obtain~\eqref{eqnqRep}, it remains to use~\eqref{eqnDGreen} and the large $k$ asymptotics  
\begin{align*}
   2\I k \biggl(\frac{f_+'(k,x)}{f_+(k,x)} - \frac{f_-'(k,x)}{f_-(k,x)}\biggr)^{-1} = 1 + \frac{q(x)}{2k^2} + \oo\biggl(\frac{1}{k^2}\biggr),
\end{align*}
 which are valid for continuous potentials~\cite{at81,dale91}. 
\end{proof}

\begin{remark}
 The representation of reflectionless integrable potentials via solutions of coupling problems in this section together with uniqueness of solutions to coupling problems established in~\cite{UniSolCP} provides an alternative proof of the fact that eigenvalues and norming constants uniquely determine a reflectionless integrable potential, which was originally proved in~\cite{hrmemy21}. 
\end{remark}

\section{Generalized soliton solutions}\label{secGSS}

The recent solution of the inverse spectral problem for reflectionless integrable potentials in~\cite{hrmemy21} provides a bijection between such potentials and a certain class of spectral data.  
By means of this correspondence, the time evolution of the spectral data 
\begin{align}
  \kappa_n(t) & = \kappa_n(0),  & c_n(t) & = \E^{8\kappa_n^3 t}c_n(0), & \gamma_{\pm,n}(t) & = \E^{\pm 4\kappa_n^3 t} \gamma_{\pm,n}(0),
\end{align}
 gives rise to classical solutions of the Korteweg--de Vries equation with arbitrary reflectionless integrable initial data; see~\cite[Section~6]{hrmemy21}. 
 We will refer to this kind of solutions as {\em generalized soliton solutions} in the following.
 In the special case when the underlying Schr\"odinger operator has only finitely many negative eigenvalues (equivalently, when the potential has sufficiently strong spatial decay), they are precisely the well-known $N$-soliton solutions. 

It follows readily from our considerations in Section~\ref{secRI} that generalized soliton solutions $q$ can be represented in terms of solutions to symmetric coupling problems, where the set $\sigma$ defined by 
 \begin{align}
   \sigma = \biggl\{-\frac{1}{\kappa_n},\,\frac{1}{\kappa_n}\,\bigg|\, n=1,\ldots,N\biggr\}.
 \end{align}
 is independent of time $t$ because so are the $\kappa_n$:
 {\em For each given point $x\in\R$ and time $t\in\R$ one has  
\begin{align}\label{eqnqxtRep}
  q(x,t) = - \frac{d^2}{dz^2} \frac{\Phi_-(z)\Phi_+(z)}{W(z)}\bigg|_{z=0},
\end{align}
where the pair $(\Phi_-,\Phi_+)$ is the solution of the symmetric coupling problem with data $\eta_{x,t}$ given by}
\begin{align}
  \eta_{x,t}\biggl(\frac{1}{\kappa_n}\biggr) = c_n(0) \E^{8\kappa_n^3 t -2\kappa_n x}. 
\end{align}
In combination with the results on the asymptotic behavior of solutions to symmetric coupling problems in Theorem~\ref{thmAA} and Corollary~\ref{corAA}, this representation allows us to derive long-time asymptotics for generalized soliton solutions next.   
    
\begin{theorem}\label{thmLT}
  Let $q$ be a generalized soliton solution of the Korteweg--de Vries equation with reflectionless integrable initial data. 
 Then one has the asymptotics  
 \begin{align}\label{eqnLTu}
  q(x,t) = \sum_{n=1}^N \frac{-2\kappa_n^2}{\cosh^2(\kappa_n x-4\kappa_n^3 t - \xi_n)} + \oo(1)  
 \end{align}
 uniformly in $x\in\R$ as $t\rightarrow\infty$, where the phase shifts $\xi_n$ are given by
 \begin{align}
   \xi_n = \frac{1}{2}\log|c_n(0)| - \frac{1}{2} \sum_{j=1}^{n-1} \log\biggl(\frac{\kappa_j+\kappa_n}{\kappa_j-\kappa_n}\biggr) +\frac{1}{2} \sum_{j=n+1}^N \log\biggl(\frac{\kappa_n+\kappa_j}{\kappa_n-\kappa_j}\biggr).
 \end{align}
\end{theorem}

\begin{proof}
  Suppose that there is an $\varepsilon>0$, a sequence $t_k\rightarrow\infty$ and $x_k\in\R$ such that 
  \begin{align}\label{eqnuLTeps}
    \Biggl| q(x_k,t_k) + \sum_{n=1}^N \frac{2\kappa_n^2}{\cosh^2(\kappa_n x_k-4\kappa_n^3 t_k - \xi_n)} \Biggr| > \varepsilon
  \end{align}
  for all $k\in\N$. 
  By possibly passing to a suitable subsequence if necessary, we may assume that $x_k/t_k$ converges to $-\infty$, $+\infty$ or some $c\in\R$.
  In the cases when this limit is $-\infty$, $+\infty$ or $c$ with $c\not=4\kappa_n^2$ for all $n=1,\ldots,N$, one has that $\eta_{x_k,t_k}(\lambda)$ converges to zero or infinity for each $\lambda\in\sigma\cap(0,\infty)$. 
  We then infer from Corollary~\ref{corAA} and the representation in~\eqref{eqnqxtRep} that $q(x_k,t_k) \rightarrow 0$. 
  Since by dominated convergence we also have that 
  \begin{align*}
     \sum_{n=1}^N \frac{2\kappa_n^2}{\cosh^2(\kappa_n x_k-4\kappa_n^3 t_k - \xi_n)} & \rightarrow 0
  \end{align*}
  in these cases, this contradicts~\eqref{eqnuLTeps}. 
  Otherwise, when the limit $c$ is equal to $4\kappa_n^2$ for some $n$, one has that 
  \begin{align*}
    \eta_{x_k,t_k}\biggl(\frac{1}{\kappa_j}\biggr) \rightarrow \begin{cases} 0, & j>n, \\ \infty, & j<n. \end{cases}
  \end{align*}
   In this case, it follows from Theorem~\ref{thmAA} that 
  \begin{align*}
    q(x_k,t_k) & =  \frac{-2\kappa_n^2}{\cosh^2\bigl(\frac{1}{2}\log|\hat{\eta}_k|\bigr)} + \oo(1),
  \end{align*}
  where the $\hat{\eta}_k$ are given by 
 \begin{align*}
  \hat{\eta}_{k} = c_n(0)  \E^{8\kappa_n^3 t_k -2\kappa_n x_k} \prod_{j=1}^{n-1} \frac{\kappa_n - \kappa_j}{\kappa_n + \kappa_j} \prod_{j=n+1}^N \frac{\kappa_n + \kappa_j}{\kappa_n - \kappa_j}.
 \end{align*}
 As this readily turns into the asymptotics 
  \begin{align*}
    q(x_k,t_k) & =  \frac{-2\kappa_n^2}{\cosh^2(\kappa_n x_k-4\kappa_n^3 t_k - \xi_n)} + \oo(1)
  \end{align*}
  and one clearly also has 
  \begin{align*}
    \sum_{\substack{j=1\\j\not=n}}^N \frac{2\kappa_j^2}{\cosh^2(\kappa_j x_k-4\kappa_j^3 t_k - \xi_j)} & \rightarrow 0 
  \end{align*}
  in this case, we arrive at a contradiction to~\eqref{eqnuLTeps} again.   
\end{proof}


\begin{thebibliography}{XX}

\bibitem{at81}
F.\ V.\ Atkinson, {\em On the location of the Weyl circles}, Proc.\ Roy.\ Soc.\ Edinburgh Sect.\ A {\bf 88} (1981), no.~3-4, 345--356. 

\bibitem{bo03}
A.\ Borodin, {\em Discrete gap probabilities and discrete Painlev\'e equations}, Duke Math.\ J.\ {\bf 117} (2003), no.~3, 489--542.

\bibitem{bo00}
A.\ Borodin, {\em Riemann--Hilbert problem and the discrete Bessel kernel}, Internat.\ Math.\ Res.\ Notices (2000), no.~9, 467--494.

\bibitem{dale91}
A.\ A.\ Danielyan and B.\ M.\ Levitan, {\em Asymptotic behavior of the Weyl--Titchmarsh $m$-function}, Math.\ USSR-Izv.\ {\bf 36} (1991), no.~3, 487--496. 

\bibitem{deitzh93}
P.\ A.\ Deift, A.\ R.\ Its and X.\ Zhou, {\em Long-time asymptotics for integrable nonlinear wave equations}, in {\em Important developments in soliton theory}, 181--204, Springer Ser.\ Nonlinear Dynam., Springer, Berlin, 1993.

\bibitem{detr79}
P.\ Deift and E.\ Trubowitz, {\em Inverse scattering on the line}, Comm.\ Pure Appl.\ Math.\ {\bf 32} (1979), no.~2, 121--251.

\bibitem{dezh93}
P.\ Deift and X.\ Zhou, {\em A steepest descent method for oscillatory Riemann--Hilbert problems. Asymptotics for the MKdV equation}, Ann.\ of Math.\ (2) {\bf 137} (1993), no.~2, 295--368. 

\bibitem{UniSolCP}
J.\ Eckhardt, {\em Unique solvability of a coupling problem for entire functions}, Constr.\ Approx.\ {\bf 49} (2019), no.~1, 123--148.

\bibitem{CouplingProblem}
J.\ Eckhardt and G.\ Teschl, {\em A coupling problem for entire functions and its application to the long-time asymptotics of integrable wave equations}, Nonlinearity {\bf 29} (2016), no.~3, 1036--1046.

\bibitem{gekazh92}
F.\ Gesztesy, W.\ Karwowski and Z.\ Zhao, {\em Limits of soliton solutions}, Duke Math.\ J.\ {\bf 68} (1992), no.~1, 101--150.

\bibitem{grte09}
K.\ Grunert and G.\ Teschl, {\em Long-time asymptotics for the Korteweg--de Vries equation via nonlinear steepest descent}, Math.\ Phys.\ Anal.\ Geom.\ {\bf 12} (2009), no.~3, 287--324.

\bibitem{hrmemy21}
R.\ Hryniv, B.\ Melnyk and Y.\ Mykytyuk, {\em Inverse scattering for reflectionless Schr\"odinger operators with integrable potentials and generalized soliton solutions for the KdV equation}, Ann.\ Henri Poincar\'e {\bf 22} (2021), no.~2, 487--527.

\bibitem{hrmy21}
R.\ O.\ Hryniv and Ya.\ V.\ Mykytyuk, {\em On the first trace formula for Schr\"odinger operators}, J.\ Spectr.\ Theory {\bf 11} (2021), no.~2, 489--507.

\bibitem{humBre16}
 I.\ Hur, M.\ McBride and C.\ Remling, {\em The Marchenko representation of reflectionless Jacobi and Schr\"odinger operators}, Trans.\ Amer.\ Math.\ Soc.\ {\bf 368} (2016), no.~2, 1251--1270.

\bibitem{le80}
B.\ Ja.\ Levin, {\em Distribution of zeros of entire functions}, Revised edition, Translations of Mathematical Monographs, 5, American Mathematical Society, Providence, R.I., 1980. 

\bibitem{ma11}
V.\ A.\ Marchenko, {\em Sturm--Liouville operators and applications}, Revised edition, AMS Chelsea Publishing, Providence, RI, 2011. 

\bibitem{ma91}
V.\ A.\ Marchenko, {\em The Cauchy problem for the KdV equation with nondecreasing initial data}, in {\em What is integrability?}, 273--318, Springer Ser.\ Nonlinear Dynam., Springer-Verlag, Berlin, 1991. 

\bibitem{no92}
V.\ Yu.\ Novoksh\"enov, {\em Reflectionless potentials and soliton series of the KdV equation}, Theoret.\ and Math.\ Phys.\ {\bf 93} (1992), no.~2, 1279--1291.

\bibitem{ta72}
S.\ Tanaka, {\em On the $N$-tuple wave solutions of the Korteweg--de Vries equation}, Publ.\ Res.\ Inst.\ Math.\ Sci.\ {\bf 8} (1972/73), 419--427.

\bibitem{wato72}
M.\ Wadati and M.\ Toda, {\em The exact $N$-soliton solution of the Korteweg--de Vries equation}, Phys.\ Soc.\ Japan {\bf 32} (1972), 1403--1411. 

\bibitem{we96}
T.\ Weidl, {\em On the Lieb--Thirring constants $L_{\gamma,1}$ for $\gamma\geq 1/2$}, Comm.\ Math.\ Phys.\ {\bf 178} (1996), no.~1, 135--146. 

\bibitem{ya10}
D.\ R.\ Yafaev, {\em Mathematical Scattering Theory}, Math.\ Surveys Monogr., 158, American Mathematical Society, Providence, RI, 2010. 

\end{thebibliography}
\end{document}